\documentclass[12pt]{article}
\usepackage{amscd, amssymb, latexsym, amsmath}

\newtheorem{thm}{Theorem}

\newtheorem{pro}{Proposition}

\newtheorem{rem}{Remark}

\newenvironment{proof}
{\noindent {\em Proof}} {\hfill $\Box$}

\numberwithin{thm}{section} \numberwithin{cor}{section}
\numberwithin{pro}{section}
\numberwithin{lem}{section} \numberwithin{equation}{section}

\newcommand{\R}{\mathbb R}

\numberwithin{equation}{section}

\newcommand{\bp}  {\begin{proof}}
\newcommand{\ep}  {\end{proof}}
\newcommand{\sfrac}[2]{{\textstyle \frac{#1}{#2}}}

\newcommand{\re}{\mathbb R}
\newcommand{\sph}{\mathbb S}

\begin{document}
\author{John Loftin* and Mao-Pei Tsui**}
\title{Limits of Solutions to a Parabolic Monge-Amp\`ere Equation}
\date{November 13, 2007}
\maketitle \leftline{*Department of Mathematics and Computer Science
,} \leftline{\text{ }\,Rugters University, Newark, NJ  07102 }
 \centerline{email: loftin@rutgers.edu }
\leftline{**Department of Mathematics, University of Toledo, Toledo,
OH 43606,USA} \leftline{\text{  }\,\,\,Taida Institute for
Mathematical Sciences, Taipei, Taiwan}\centerline{email:
Mao-Pei.Tsui@Utoledo.edu}

\section{Introduction}

In \cite{loftin-tsui06}, we study solutions to the affine normal
flow for an initial hypersurface $\mathcal L\subset \R^{n+1}$ which
is a convex, properly embedded, noncompact hypersurface.  The method
we used was to consider an exhausting sequence $\mathcal L_i$ of
smooth, strictly convex, compact hypersurfaces so that each
$\mathcal L_{i}$ is contained in the convex hull of $\mathcal
L_{i+1}$ for each $i$, and so that $\mathcal L_i \to \mathcal L$
locally uniformly.  If the compact $\mathcal L_i$ is the initial
hypersurface, the affine normal flow $\mathcal L_i(t)$ is
well-defined for all time $t$ from 0 to the extinction time $T_i$
\cite{chow85}. Then for all positive $t$, we define the affine
normal flow for initial hypersurface $\mathcal L$ as a limit
$\mathcal L(t) = \lim_{i\to\infty}\mathcal L_i(t)$. Ben Andrews
extensively studies the affine normal flow for compact initial
hypersurfaces \cite{andrews96,andrews00}.

The method of proof in \cite{loftin-tsui06} is to consider the
support functions $s_{\mathcal L_i} = s_i$ and to take the limit as
$i\to\infty$.  For each $Y\in\R^{n+1}$, the support function is
defined by
$$ s(Y)= s_{\mathcal L}(Y) = \sup_{x\in\mathcal L} \langle x,Y \rangle,$$
for $\langle\cdot,\cdot\rangle$ the Euclidean inner produce on
$\R^{n+1}$. It is immediate that $s$ is a convex function of
homogeneity one on $\R^{n+1}$.  The homogeneity property means that
it suffices to study the behavior of $s$ when restricted to the unit
sphere $\mathbb S^n\subset \R^{n+1}$. Also, $s$ restricted to an
affine hyperplane not touching the origin in $\R^{n+1}$ determines
$s$ on a half-space of $\R^{n+1}$.  We consider $s$ in this setting
primarily: If $Y=(y,-1)$ for $y\in\R^n$, then $s$ evolves under the
affine normal flow by
 \begin{equation} \label{ma-eq}
 \frac{\partial s}{\partial t} = - \left (\det
\frac{\partial ^2s }{\partial y^i \partial y^j} \right) ^{-
\frac1{n+2}}. \end{equation}
 Note that this setting of considering the restriction of $s$
to the hyperplane $\{y^{n+1}=1\}$ has its roots in the
Minkowski problem (see Cheng-Yau \cite{cheng-yau76}).

In the present paper, we consider our previous result primarily
the point of view of Equation (\ref{ma-eq})---in other words,
from more of a classical PDE point of view as opposed to the
largely tensorial point of view in \cite{loftin-tsui06}.  Also,
to the extent possible, we phrase the proofs in analytic terms,
and try not to rely too much on the affine geometry. In
particular, consider the support function $s_i$ of $\mathcal
L_i$. Then as $i\to\infty$, $s_i(Y)$ increases to the limit
$s(Y)$ for all $Y\in\R^{n+1}$ (this follows by the exhaustion
property of $\mathcal L_i\to \mathcal L$). The noncompactness
of $\mathcal L$ implies that $s$ is equal to $+\infty$ on at
least a half-space of $\R^{n+1}$. Let $\mathcal D^\circ (s)$ be
the largest open subset of $\R^{n+1}$ on which $s<\infty$.
($\mathcal D^\circ(s)$ is then the interior of the domain of
$s$, which is defined by $\mathcal D(s) = \{Y:s(Y)<+\infty\}$.)
Since $\mathcal D^\circ (s)$ is contained in an open half-space
of $\R^{n+1}$, we may (by choosing new coordinates if
necessary) restrict to the affine hyperplane
$\{Y=(y,-1):y\in\R^n\}$ and consider the limit $s_i\nearrow s$.

We make the following nondegeneracy assumptions about $\mathcal L$
and thus $s$.  First, assume that $\mathcal L$ does not contain any
lines. This is equivalent to
\begin{equation} \label{domain-open}
\mathcal D^\circ(s) \neq \emptyset
\end{equation} (see e.g.\
Rockafellar \cite{rockafellar}).  Also assume that $\mathcal L$ is a
hypersurface, and not a lower-dimensional set. So, in particular,
the convex hull $ \hat{\mathcal L}$ has nonempty interior, and thus
contains a small ball $B_\epsilon(P)$.  Thus $s=s_{\mathcal L} =
s_{\hat{\mathcal L}} \ge s_{B_\epsilon(P)}$, and there are $P\in
\R^{n+1}$ and $\epsilon>0$ so that for all $Y\in\R^{n+1}$,
\begin{equation} \label{convex-s-Y} s(Y) \ge
\epsilon|Y| + \langle P, Y \rangle
 \end{equation}
For $Y=(y,-1)$, this assumption becomes that there are $\epsilon
>0$, $p\in\R^n$ and $c\in R$ so that for all $y\in\R^n$,
\begin{equation} \label{convex-s}
s(y) \ge \epsilon\sqrt{|y|^2+1} + \langle p,y\rangle - c
\end{equation}
Also note that equation (\ref{convex-s-Y}) may be computed using the
following useful transformation law for the support function: If
$A\in \mathbf{GL}(n+1,\R)$ and $b\in \R^{n+1}$, then
\begin{equation} \label{affine-transform}
 s_{A\mathcal L + b} (Y) = s_{\mathcal L} (A^\top Y) + \langle b,
 Y\rangle.
\end{equation}
This rule is particularly useful, since the affine normal flow
is invariant under all affine volume-preserving maps of
$\R^{n+1}$. Note also that (\ref{affine-transform}) is
equivalent to a projective transformation of $s$ when
restricted to $\{y^{n+1}=-1\}$.

In terms of the support function functions, we consider
$s_i\nearrow s$, where the $s_i\!:\re^{n+1}\to\re$ are all
convex functions of homogeneity one on $\re^{n+1}\setminus
\{0\}$ which are smooth and strictly convex on each affine
hyperplane in $\re^{n+1}$ which does not pass through the
origin.  Then the affine normal flow $s_i(t)$ may be defined by
solving (\ref{ma-eq}) on affine coordinate hyperplanes
$\{y^i=\pm1\}$ and patching together the solutions. More
simply, $s_i|_{\mathbb S^n}$ solves a parabolic equation, and
thus we have existence and uniqueness for a short time (as
noted by Chow \cite{chow85} originally). Then we let $s_i\to s$
pointwise everywhere in $\re^{n+1}$, given the nondegeneracy
assumptions (\ref{domain-open}) and (\ref{convex-s}) and as
well that the interior of the domain $\mathcal D^\circ (s)$ is
contained in the half-space $\{y^{n+1}<0\}$.

Now for the affine normal flow, $s_i(t) \nearrow s(t)$ as
$i\to\infty$. On $\mathcal D^\circ (s)$, this is an increasing limit
of smooth strictly convex functions (and so $s(t)$ is Lipschitz a
priori). Our problem is then to examine which properties of the
solutions $s_i(t)$ to (\ref{ma-eq}) survive in the limit $s_i(t)
\nearrow s(t)$ on $\mathcal D^\circ(s)$. This will determine the
regularity properties of $s(t)$.  In particular, there are locally
uniform spacelike $C^{0,1}$ estimates on $s_i$ on $\mathcal
D^\circ(s)$ just by convexity. Uniform spacelike $C^2$ and
ellipticity estimates follow by a global speed estimate of Andrews
\cite{andrews00} which survives in the limit as $s_i\to s$ and a
local Pogorelov-type estimate of Guti\'errez-Huang
\cite{gutierrez-huang98}.  We also use a barrier due to Calabi
\cite{calabi72} to ensure we can apply Guti\'errez-Huang's estimate
to get locally uniform spacelike $C^2$ estimates on $s_i$ for all
positive $t$. Then Evans-Krylov theory applies to get locally
uniform parabolic $C^{2+\alpha,1+\alpha/2}$ estimates and standard
bootstrapping implies local $C^\infty$ convergence of $s_i\to s$ for
positive time $t$.

There is also an important estimate of Ben Andrews \cite{andrews96}
on $|C|^2$ associated to $s_i$ the support function  a compact,
smooth, strictly convex hypersurfaces $\mathcal L_i$, for a tensor
$C$ called the cubic form. This estimate shows that for any ancient
solution to the affine normal flow, $|C|^2=0$, which implies by a
classical theorem of Berwald that $\mathcal L$ is a quadratic
hypersurface.  In Section \ref{quadric-sec} below, we reproduce this
classical theorem from the point of view of the support function
$s$.

The first author is grateful for the support of the NSF under Grant
DMS0405873, and to the organizers of the 2007 Conference on
Geometric Analysis in Taipei, for the opportunity to speak and for
their kind hospitality during the conference. The second author
wishes to express his gratitude to Taida Institute for Mathematical
Sciences for providing an excellent research environment and the
support of C.S. Lin and Y.I. Lee. We would both like to thank D.H.\
Phong for his support and for his suggestion to write up our results
in this way.

\section{Support Function}
In this section, we compute some of the basic quantities of affine
differential geometry in terms of the support function $s$. In the
end of this section, we show that (\ref{ma-eq}) is equivalent to
affine normal flow.

 Let $F$ be a smooth embedding of
a strictly convex hypersurface in terms of an extended Gauss map.
This means $F=F(Y)$ for any vector $Y$ equal to a negative multiple
of the inward-pointing unit normal vector $\nu$ to the image of $F$.
So $F$ is a function from an collection of open rays in
$\re^{n+1}\setminus \{0\}$ to $\re^{n+1}$ which is homogeneous of
degree 0.  In particular, we have
$$ s(Y) = \langle F,Y \rangle. $$

The affine normal $\xi$ is a transverse vector field to the image of
$F$ which is invariant under the action of all volume-preserving
affine maps in $\re^{n+1}$.  We recall the basic tensors and
structure equations of affine differential geometry: For each $y$ in
the domain of $F$, consider the basis $F_1,\dots,F_n,\xi$ of
$\re^{n+1}$, write the derivatives of these basis elements in terms
of the same basis:
\begin{eqnarray*}
F_{ij} &=& (\Gamma_{ij}^k + C_{ij}^k)F_k + g_{ij} \xi, \\
\xi_i &=& - A_i^j F_j.
\end{eqnarray*}
Here $g_{ij}$ the affine metric, or affine second fundamental form,
is positive definite for strictly convex hypersurfaces;
$\Gamma_{ij}^k$ is its Levi-Civita connection; $C_{ij}^k$ is the
cubic form; and $A_i^j$ is the affine curvature, or affine shape
operator.

Now we derive the formula for the cubic form $C^k_{ij}$ in terms of
the support function $s$:

Under the extended Gauss map, the inward-pointing Euclidean unit
normal $\nu$ satisfies
\begin{equation}\label{nu} \nu = - \frac{Y}{|Y|} = \frac{(-y^1, \dots ,-y^n,
1)} {\sqrt{ 1 + |y|^2}},
\end{equation}
and the (Euclidean) second fundamental form is given by
\begin{equation}\label{h_ij}
h_{ij} = \frac {s_{ij}} {\sqrt{1+|y|^2}}, \qquad
h^{ij}=\sqrt{1+|y|^2}s^{ij}.\end{equation}
 The scalar function $\phi$ is defined to be $ (\det
h_{ij})^{\frac 1 {n+2}} (\det \bar g_{ij})^{-\frac 1 {n+2}}$ for
$\bar g_{ij}=\langle F_i,F_j\rangle$ the induced metric from
Euclidean $\re^{n+1}$. We compute (using the formula for $\bar
g_{ij}$ below)
 $$ \phi =
(1+|y|^2)^{-\frac{1}{2} }D^{-\frac{1}{n+2}}, $$ for $D = \det
s_{ij}$, and
$$\phi_k =-y^k
(1+|y|^2)^{-\frac{3}{2}
}D^{-\frac{1}{n+2}}-\frac{1}{n+2}(1+|y|^2)^{-\frac{1}{2} }
D^{-\frac{1}{n+2}}(\ln D)_k,$$ where $(\ln D)_k = s^{pq}s_{pqk}$. We
also define the vector field $Z^i$ by
\begin{equation}
\begin{split}
&  Z^i\\
 = & -h^{ik} \phi_k\\
 = & -\sqrt{1+|y|^2}s^{ik}[-y^k
(1+|y|^2)^{-\frac{3}{2} }
D^{-\frac{1}{n+2}}-\frac{1}{n+2}(1+|y|^2)^{-\frac{1}{2} }D^{-\frac{1}{n+2}}(\ln D)_k] \\
= &  D^{-\frac{1}{n+2}}[(1+|y|^2)^{-1
}s^{ik}y^k+\frac{1}{n+2}s^{ik}(\ln D)_k]
\end{split}
 \end{equation}

  $$ F
= (s_1,\dots, s_n, (s_l y^l) - s), $$
 \begin{equation}\label{F_i} F_i
= (s_{1i},\dots, s_{ni}, s_{li} y^l)\end{equation}

\begin{equation}\label{F_ij}F_{ij} = (s_{1ij},\dots, s_{nij}, s_{lij} y^l+s_{ij})
\end{equation}

In terms of the scalar function $\phi$ and the vector field $Z^i$
defined above, we define the affine normal $\xi$ as
\begin{equation} \label{an}
\begin{split} & \xi\\
 = \ & \phi \nu + Z^i F_i \\
 =  \ & (1+|y|^2)^{-\frac{1}{2} }D^{-\frac{1}{n+2}}\cdot \frac{(-y^1, \dots ,-y^n, 1)} {\sqrt{ 1 +
|y|^2}}\\
+ \ &  D^{-\frac{1}{n+2}}[(1+|y|^2)^{-1
}s^{ik}y^k+\frac{1}{n+2}s^{ik}(\ln D)_k](s_{1i},\dots,
s_{ni}, s_{li} y^l)\\
=  \ & D^{-\frac{1}{n+2}}(1+|y|^2)^{-1 }\cdot
(-y^1, \dots ,-y^n, 1) \\
+ \ &  D^{-\frac{1}{n+2}}(1+|y|^2)^{-1}(y^1, \dots ,y^n,
|y|^2)\\
+ \ & \frac{ D^{-\frac{1}{n+2}}}{n+2}((\ln D)_1,\dots, (\ln D)_n,
(\ln D)_i y^i)
\\
= \  & \frac{1}{n+2}D^{-\frac{1} {n+2}} ( (\ln D)_1, \dots,(\ln
D)_n, (n+2) +
(\ln D)_i y^i ).\\
\end{split}
 \end{equation}
The affine normal $\xi$ is invariant under volume-preserving affine
actions on $\re^{n+1}$. The affine metric (also called the affine
second fundamental form) $g_{ij}$ is invariant under the same group,
and is given by $g_{ij} = \phi^{-1} h_{ij}$. So compute

$$ g_{ij} = D^{\frac{1}{n+2}}
s_{ij}. $$

In terms of $s=s(y,-1)=s(Y)$, the embedding $F$ is given by

 $$ F
= (s_1,\dots, s_n, (s_i y^i) - s), $$

$$ g_{ij} = (\det s_{k\ell})^{\frac{1}{n+2}}
s_{ij} $$

$$ \partial_mg_{ij} = (\det s_{k\ell})^{\frac{1}{n+2}}
(\frac{1}{n+2}s^{pq}s_{pqm} s_{ij} +s_{ijm})$$
$$ g^{ij} = (\det s_{k\ell})^{-\frac{1}{n+2}}
s^{ij} $$

\begin{equation}
\begin{split} &  \label{Christ}\Gamma^k_{ij}\\
= \ & \sfrac12 g^{kl}(\partial_i g_{j\ell} + \partial_j g_{i\ell} -
\partial_\ell g_{ij})\\
= \ & \sfrac12  s^{kl}(\frac{1}{n+2}s^{mp}s_{mpi} s_{jl} +s_{jli} +
\frac{1}{n+2}s^{mp}s_{mpj}
s_{il} +s_{ilj} - \frac{1}{n+2}s^{mp}s_{mpl} s_{ij} - s_{ijl} )\\
= \ & \frac12 ( \frac1{n+2} s^{mp}s_{mpj}\delta^k_i + \frac1{n+2}
s^{mp}s_{mpi} \delta^k_j + s^{k\ell}s_{ij\ell} - \frac1{n+2}
s^{k\ell} s^{mp} s_{mp\ell} s_{ij} )\\
= \ & \frac12 ( \frac1{n+2}(\ln D)_j\delta^k_i + \frac1{n+2} (\ln
D)_i \delta^k_j + s^{k\ell}s_{ij\ell} - \frac1{n+2} s^{k\ell} (\ln
D)_{\ell} s_{ij} ),
\end{split}
 \end{equation}
where we define $D=\det s_{ij}$. Now compute the metric induced from
the Euclidean metric $\bar g_{ij}$.
\begin{eqnarray*}
\bar g_{ij} &=& \left\langle \frac{\partial F}{\partial y^i},
\frac{\partial F}{\partial y^j} \right\rangle \\
&=& \sum_{k,l=1}^n \frac{\partial^2s}{\partial y^i \partial y^k}
(y^ky^l + \delta^{kl}) \frac{\partial^2s}{\partial y^j \partial
y^l}, \\
\bar g^{ij} &=&  \sum_{k,l=1}^n s^{ik} (-\frac{y^ky^l}{1+|y|^2} +
\delta_{kl})s^{lj}, \text{where }  s^{mn}
\text{ is the inverse of } s_{ij}, \\
\det{\bar g_{ij}} &=& \det \left(\frac{\partial^2s}{\partial y^i
\partial y^k} \right) \det (y^ky^l + \delta^{kl}) \det \left(
\frac{\partial^2s}{\partial y^j \partial y^l}\right) \\
&=& (1+|y|^2)\det \left(\frac{\partial^2 s}{\partial y^i
\partial y^j} \right)^2.
\end{eqnarray*}

Recall that $ \xi =  \phi \nu + Z^k F_k $, $h_{ij}= \phi g_{ij}$ and
$$F_{ij}= g_{ij}(\phi \nu + Z^k F_k)+ (\Gamma^k_{ij}+
C^k_{ij})F_k.$$ So $$\langle F_{ij}, F_l \rangle =g_{ij}Z^k
\overline{g}_{kl}
+\Gamma^k_{ij}\overline{g}_{kl}+C^k_{ij}\overline{g}_{kl},$$
$$\langle F_{ij}, F_l \rangle\overline{g}^{lm} =g_{ij}Z^m
+\Gamma^m_{ij}+C^m_{ij},$$
$$C^m_{ij}=\langle F_{ij}, F_l \rangle\overline{g}^{lm} -g_{ij}Z^m
-\Gamma^m_{ij}$$ and, lowering the index by the affine metric
$C_{ijk} = C_{ij}^l g_{lk}$,
\begin{equation}
\label{cubic1}C_{ijk}=\langle F_{ij}, F_l
\rangle\overline{g}^{lm}g_{mk} -g_{ij}Z^mg_{mk}
-\Gamma^m_{ij}g_{mk}.
\end{equation}

First, we compute \begin{align*} &-g_{ij}Z^k \\
= \ &-(\det s_{k\ell})^{\frac{1}{n+2}} s_{ij}(\det
s_{rs})^{-\frac{1}{n+2}}[(1+|y|^2)^{-1
}s^{kl}y^l+\frac{1}{n+2}s^{kl}(\sum_{p,q}s^{pq}s_{pql})] \\
= \ & - s_{ij}[(1+|y|^2)^{-1 }s^{kl}y^l+\frac{1}{n+2}s^{kl}(\ln
D)_l]
\end{align*}

So  \begin{equation}
\begin{split} \label{C2}
 & -g_{ij}Z^lg_{lk}\\
= \ & - s_{ij}[(1+|y|^2)^{-1
}s^{lm}y^m+\frac{1}{n+2}s^{lm}(\sum_{p,q}s^{pq}s_{pqm})](\det
s_{rs})^{\frac{1}{n+2}} s_{lk}\\
= \ & -s_{ij}(\det s_{rs})^{\frac{1}{n+2}}[(1+|y|^2)^{-1
}y^k+\frac{1}{n+2}(\sum_{p,q}s^{pq}s_{pqk})]\\
= \ & -(\det s_{rs})^{\frac{1}{n+2}}[s_{ij}(1+|y|^2)^{-1
}y^k+\frac{s_{ij}}{n+2}(\ln D)_k]\\
\end{split}
 \end{equation}

Now, we compute \begin{align*} & \overline{g}^{lm}g_{mk}\\
= \ & s^{lp} (-y^py^q(1+y^2)^{-1}+ \delta_{pq})s^{qm}
D^{\frac{1}{n+2}} s_{mk}\\
= \ & D^{\frac{1}{n+2}} (-s^{lp}y^py^k(1+y^2)^{-1}+
s^{lk})\end{align*} and

\begin{align*}
 & \langle F_{ij}, F_l \rangle\\
 = \  & \langle (s_{1ij},\dots, s_{nij}, s_{rij} y^r+s_{ij}), (s_{1k},\dots, s_{nk}, s_{mk} y^m)
 \rangle\\
= \  & \sum_p s_{pij}s_{pl}+ s_{rij} y^rs_{ml} y^m+s_{ij} s_{ml}
y^m.
\end{align*}

So

\begin{equation}
\begin{split} \label{C1}
 & \langle F_{ij}, F_l \rangle \overline{g}^{lm}g_{mk}\\
 = \  &  D^{\frac{1}{n+2}} ( \sum_p s_{pij}s_{pl}+ s_{rij} y^rs_{ml} y^m+s_{ij} s_{ml}
y^m) (-s^{lq}y^qy^k(1+y^2)^{-1}+ s^{lk})\\
= \  &  D^{\frac{1}{n+2}} ( -\sum_p s_{pij}y^py^k(1+y^2)^{-1}-
s_{rij} |y|^2(1+y^2)^{-1}y^r
y^k-s_{ij}|y|^2(1+y^2)^{-1}  y^k \\
+ \ & s_{kij}+ s_{rij} y^ry^k+s_{ij}
y^k))\\
= \  &  D^{\frac{1}{n+2}} ( -\sum_p
s_{pij}y^py^k-s_{ij}|y|^2(1+y^2)^{-1}  y^k  +  s_{kij}+ s_{rij}
y^ry^k+s_{ij}
y^k))\\
\end{split}
 \end{equation}

\begin{equation}
\begin{split} & \label{C3} \Gamma^m_{ij}g_{mk}\\
= \ & \frac12 ( \frac1{n+2}(\ln D)_j\delta^m_i + \frac1{n+2} (\ln
D)_i \delta^m_j + s^{m\ell}s_{ij\ell} - \frac1{n+2} s^{m\ell} (\ln
D)_{\ell} s_{ij} )D^{\frac{1}{n+2}} s_{mk}\\
= \ & \frac{D^{\frac{1}{n+2}}}{2} ( \frac1{n+2}(\ln D)_js_{ik} +
\frac1{n+2} (\ln D)_i  s_{jk}+ s_{ijk} - \frac1{n+2} (\ln D)_{k}
s_{ij} )
\end{split}
 \end{equation}

From (\ref {C1}), (\ref{C2}), (\ref {C3}) and (\ref {cubic1}), we
have
 \begin{equation}\begin{split}\label{cubic} & C_{ijk}\\
= \  &  D^{\frac{1}{n+2}} \Big( -\sum_p s_{pij}y^py^k-s_{ij}|y|^2(1+y^2)^{-1}  y^k \\
+ \ & s_{kij}+ s_{rij} y^ry^k+s_{ij} y^k)\Big)-(\det
s_{rs})^{\frac{1}{n+2}}[s_{ij}(1+|y|^2)^{-1
}y^k+\frac{s_{ij}}{n+2}(\ln D)_k]\\
- \ & \frac{D^{\frac{1}{n+2}}}{2} ( \frac1{n+2}(\ln D)_js_{ik} +
\frac1{n+2} (\ln D)_i  s_{jk}+ s_{ijk} - \frac1{n+2} (\ln D)_{k}
s_{ij} )\\
= \  &  D^{\frac{1}{n+2}}\Big[\frac{1}{2} s_{ijk}-\frac1{2(n+2)}
s_{ki}(\ln D)_j - \frac1{2(n+2)} s_{kj}(\ln
D)_i-\frac{s_{ij}}{2(n+2)}(\ln D)_k \Big]
\end{split}\end{equation}

Now we prove that (\ref{ma-eq}) is equivalent to the affine normal
flow.
\begin{pro} \label{anf}
The affine normal flow $$\frac{\partial}{\partial t}F = \xi$$ is
equivalent to the evolution of the support function
$$ \frac{\partial s}{\partial t} = - \left (\det
\frac{\partial ^2s }{\partial y^i \partial y^j} \right) ^{-
\frac1{n+2}}.$$
\end{pro}

\begin{proof}
We first compute $s_t$ from $F_t=\xi$: Recall that $ \nu = -
\frac{Y}{|Y|} = \frac{(-y^1, \dots ,-y^n, 1)} {\sqrt{ 1 + |y|^2}}$
which is independent of time in our coordinate system, since
${\partial_t}\nu=0$ (see \cite{loftin-tsui06}). Using the definition
$s= \langle F , Y \rangle$, $\xi
 = (1+|y|^2)^{-\frac12} D^{-\frac1{n+2}} \nu + Z^i F_i$ and
${\partial_t}Y=0$, we have
\begin{eqnarray*}
{\partial_t} s &=& \langle {\partial_t}  F , Y \rangle + \langle
F,{\partial_t}Y\rangle \\
&=&\left\langle \det (1+|y|^2)^{-\frac12} D^{-\frac1{n+2}} \nu +
Z^i F_i ,-(1 +|y|^2)^{\frac12} \nu \right\rangle \\
&=& -D^{-\frac1{n+2}}.
 \end{eqnarray*}

For good measure, we also compute $F_t$ from
$s_t=-D^{-\frac1{n+2}}$: Recall that the position function $F$ can
be expressed by the support function
$$ F = (s_1,\dots, s_n, (s_l y^l) - s).$$ Recall $D= \det (
\frac{\partial ^2s }{\partial y^i
\partial y^j})$. Note that
\begin{eqnarray*}
  s_t &=& -D^{-\frac1{n+2}}, \\
 s_{it} &=& s_{ti} = (-D^{-\frac1{n+2}})_{i} \\
 &=& \left[ \frac1{n+2} (\ln D)_i D^{-\frac1{n+2}}\right]
\end{eqnarray*}

 Compute
\begin{equation}\begin{split}\frac{\partial}{\partial t}F & =(s_{t1},
\dots, s_{tn}, s_{tl} y^l - s_t)\\
 & =  \frac1{n+2}  D^{-\frac1{n+2}}((\ln D)_1,\dots, (\ln D)_n, (\ln D)_l y^l
+n+2).
\end{split}\end{equation}

Recall that $$\xi =\frac{1}{n+2}D^{-\frac{1} {n+2}} ( (\ln D)_1,
\dots,(\ln D)_n, (n+2) + (\ln D)_i y^i )$$ from (\ref{an}).
Therefore $\frac{\partial}{\partial t}F = \xi$.

\end{proof}

\section{Andrews's Speed Estimate}
In this section, we repeat, for the reader's convenience, our
version of a speed estimate of Andrews \cite{andrews00}.

\begin{pro} \label{andrews-speed}
Let $s$ be the support function of a smooth strictly convex compact
hypersurface evolving under affine normal flow. If $s(Y,t)\ge r>0$
for all $Y\in\sph^n$ and $t\in [0,T]$, then
$$|\partial_t s| \le \left(C + C' t^{-\frac n{2n+2}}\right)s$$
on $\sph^n\times[0,T]$, where $C$ and $C'$ are constants only
depending on $r$ and $n$.
\end{pro}

\begin{proof}
Consider the function
 $$q = \frac{-\partial_t s}{s - r/2}.$$
We apply the maximum principle to $\log q = \log |\partial_ts| -
\log(s-r/2)$.  In particular, at a fixed time $t\in[0,T]$, consider
a point $Y\in\sph^n$ at which $q$ attains its maximum. By changing
coordinates, we may assume that this point $Y=(0,\dots,0,-1)$ is the
south pole.  Then, as in Tso \cite{tso85}, consider the coordinates
$y=(y^1,\dots,y^n)$ for $s$ restricted to the hyperplane
$\{(y^1,\dots,y^n,-1)\}$.  At $y=0$, we have for $i=1,\dots,n$
\begin{equation}
(\log q)_i =0 \quad \Longleftrightarrow \quad \frac{s_{ti}}{s_t} =
\frac{s_i}{s-r/2} \label{grad-zero}
\end{equation}
The condition for $(\log q)|_{\sph^n}$ to have a maximum at the
south pole is
\begin{equation}
 (\log q)_{ij} + (\log q)_{n+1}
\delta_{ij}\le 0 \label{neg-def}
\end{equation} as a symmetric matrix.
Here we use subscripts to denote ordinary differentiation
$f_i=\partial_{y^i}f$ and $f_t = \partial_t f$.

To compute the second term in (\ref{neg-def}), use Euler's
identities for a function of homogeneity one
$$\sum_{i=1}^{n+1} y^i s_{ti} = s_t, \qquad \sum_{i=1}^{n+1}
y^i s_i = s$$ at the point $Y=(0,\dots,0,-1)$ to conclude $s_{tn+1}
= -s_t$, $s_{n+1} = -s$,  and
$$ (\log q)_{n+1} =
\frac{r/2}{s-r/2}.
$$
For the first term in (\ref{neg-def}), compute
\begin{eqnarray*}
(\log q)_{ij} &=& \frac{s_{tij}}{s_t} - \frac{s_{ti}s_{tj}}{s_t^2} -
\frac{s_{ij}}{s-r/2} + \frac{s_is_j}{(s-r/2)^2} \\
&=& \frac{s_{tij}}{s_t}  - \frac{s_{ij}}{s-r/2}
\end{eqnarray*}
at $y=0$ by (\ref{grad-zero}).  Thus (\ref{neg-def}) becomes at
$y=0$
\begin{equation}
\label{neg-def2} \frac{r/2}{s-r/2}\delta_{ij} + \frac{s_{tij}}{s_t}
- \frac{s_{ij}}{s-r/2} \le 0.
\end{equation}

Now, we compute using the flow equation (\ref{ma-eq})
\begin{eqnarray*}
(\log q)_t &=& \partial_t \log |\partial_t s| - \partial_t
\log(s-r/2) \\
&=& -\frac1{n+2}\,\partial_t \log\det (s_{ij}) - \frac{s_t}{s-r/2} \\
&=& -\frac1{n+2} \,s^{ij} s_{tij} - \frac{s_t}{s-r/2}
\end{eqnarray*}
for $s^{ij}$ the inverse matrix of $s_{ij}$. Then (\ref{neg-def2})
implies that
\begin{eqnarray*}
(\log q)_t &\le&  \frac{r/2}{n+2} \cdot \frac{s_t}{s-r/2}
\delta_{ij} s^{ij} - \frac{2n}{n+2} \cdot \frac{s_t}{s-r/2} \\
&=& -\frac{r/2}{n+2} \,q \,\delta_{ij}s^{ij} + \frac{2n}{n+2}\, q, \\
q_t &\le& -\frac{r/2}{n+2} \,q^2\, \delta_{ij}s^{ij} +
\frac{2n}{n+2} \,q^2.
\end{eqnarray*}
Now if we let $\mu_i$ be the eigenvalues of $s^{ij}$, or
equivalently the reciprocals of the eigenvalues of $s_{ij}$, then we
see
$$ |s_t| = (\det s_{ij})^{-\frac1{n+2}} = \left(\prod_{i=1}^n \mu_i
\right)^{\frac1{n+2}} \le \left(\frac1n \sum_{i=1}^n \mu_i
\right)^{\frac n{n+2}} = \left(\frac1n \,\delta_{ij}s^{ij}
\right)^{\frac n{n+2}}$$ by the arithmetic-geometric mean
inequality.  Therefore,
$$ \delta_{ij}s^{ij} \ge n |s_t|^{\frac{n+2}n} = n q^{\frac{n+2}n}
(s-r/2)^{\frac{n+2}n} \ge n q^{\frac{n+2}n} (r/2)^{\frac{n+2}n}$$
since $s\ge r$.  And so finally, at $y=0$, and thus at any maximum
point of $q|_{\sph^n}$,
\begin{equation} \label{qt-at-max}
q_t \le -\frac{n(r/2)^{\frac{2n+2}n}}{n+2}\, q^{\frac{3n+2}n} +
\frac{2n}{n+2} \, q^2.
\end{equation}

Now define $Q(t) = \max_{Y\in \sph^n} q(Y,t)$.  Then
(\ref{qt-at-max}) implies that
$$ Q_t \le -Q^2\left(c_n r^{\frac{2n+2}n} Q^{\frac{n+2}n}
-c_n'\right)$$ for constants $c_n,c_n'$ depending only on $n$.
Therefore,
 \begin{equation} \label{Q-bound}
 Q\le \max \left\{c_n r^{-\frac{2n+2}{n+2}}, c_n' r^{-1}
t^{-\frac n{2n+2}}\right\}
 \end{equation}
for $c_n,c_n'$ new constants depending only on $n$. The result
easily follows.

\begin{rem} $Q$ may not be differentiable as a function of $t$, but the
above estimate (\ref{Q-bound}) still holds---see e.g.\ Hamilton
\cite[Section 3]{hamilton86}.
\end{rem}

\end{proof}

\section{Guti\'errez-Huang's Hessian Estimate}
Again, for the convenience of the reader, we reproduce our version
of Guti\'errez-Huang's Pogorelov-type estimate
\cite{gutierrez-huang98} for solutions to the Monge-Amp\`ere
equation.

First we define a \emph{bowl-shaped domain} in spacetime and its
\emph{parabolic boundary.}  A set $\Omega\subset\R^n\times \R$ is
bowl-shaped if there are constants $t_0 < T$ so that $$\Omega =
\bigcup_{t_0\le t\le T} \Omega_t\times \{t\},$$ where each
$\Omega_t$ is convex and $\Omega_{t_1}\subset \Omega_{t_2}$ whenever
$t_1<t_2$.  The parabolic boundary of $\Omega$ is then
$\partial\Omega \setminus (\Omega_T\times\{T\}).$

\begin{pro} \label{gh-est}
Let $s$ be a smooth solution to (\ref{ma-eq}) which is convex in
$y$, and let $\Omega$ be a bowl-shaped domain in space-time
$\R^n\times\R$ so that $s=0$ on the parabolic boundary of $\Omega$.
Let $\beta$ be any unit direction in space.

Then at the maximum point $P$ of the function $$w = |s|\,\,
\partial^2_{\beta\beta} s\,\, e^{\frac12(\partial_\beta s)^2},$$
$w$ is bounded  by a constant depending on only  $s(P)$, $\nabla
s(P)$ and $n$.
\end{pro}

\begin{proof}
Choose coordinates so that $\beta = (1,0,\dots,0)$ and so that at a
maximum point $P$ of $w$, $s_{ij}$ is diagonal (in order to bound
all second derivatives $s_{\beta\beta}$, it suffices to focus only
on the eigendirections of the Hessian of $s$).

Since $w$ is positive in $\Omega$ and $0$ on the parabolic boundary,
there is a point $P$ outside the parabolic boundary of $\Omega$ at
which $w$ assumes its maximum value.  We work with $\log w$ instead
of $w$.  Then at $P$,
$$(\log w)_i = 0, \qquad (\log w)_t \ge 0,
\qquad (\log w)_{ij} \le 0.$$
 Here we use $i,j,t$ subscripts for
partial derivatives in $y^i$, $y^j$ and $t$, and the last inequality
is as a symmetric matrix.  These equations become, at $P$,
\begin{eqnarray}
&\displaystyle\frac{s_i}s + \frac{s_{11i}}{s_{11}} + s_1 s_{1i} =
0,&
\label{grad-w-zero}\\
&\displaystyle\frac{s_t}s + \frac{s_{11t}}{s_{11}} + s_1 s_{1t} \ge
0,&
\label{t-deriv}\\
&\displaystyle \frac{s_{ij}}s - \frac{s_i s_j}{s^2} +
\frac{s_{11ij}}{s_{11}} - \frac{s_{11i}s_{11j}}{s_{11}^2} +
s_{1j}s_{1i} + s_1s_{1ij} \le 0.& \label{hess-neg}
\end{eqnarray}
To use (\ref{t-deriv}), we compute, for $D=\det s_{ij}$,
\begin{eqnarray*}
s_{1t} &=& \left(D^{-\frac1{n+2}} \right)_1 = \frac 1 {n+2} \,
D^{-\frac 1{n+2}}
s^{ij}s_{ij1},\\
 s_{11t}
&=& D^{-\frac1 {n+2}} \left[- \frac 1{(n+2)^2}(s^{ij}s_{ij1})^2  -
\frac1{n+2} \, s^{ik}s^{jl}s_{kl1}s_{ij1} + \frac1 {n+2} \,
s^{ij}s_{ij11}\right].
\end{eqnarray*}
Now plug into (\ref{t-deriv}) and divide out by $D^{-\frac1{n+2}}$
to find
 \begin{eqnarray} \nonumber
 &\displaystyle
\frac1{s_{11}} \left[ - \frac1{(n+2)^2}(s^{ij}s_{ij1})^2  - \frac 1
{n+2} \, s^{ik}s^{jl}s_{kl1}s_{ij1} + \frac 1{n+2} \, s^{ij}s_{ij11}
\right]&
\\ &\displaystyle {}-\frac1 s + s_1 (  \frac 1 {n+2} \,
s^{ij}s_{ij1}) \ge 0 & \label{u-11t}
\end{eqnarray}

The last term of the first line of (\ref{u-11t}) leads us to
contract (\ref{hess-neg}) with the positive-definite matrix $s^{ij}$
so that at $P$:
\begin{eqnarray*}
0 &\ge& s^{ij}\left(\frac{s_{ij}}s - \frac{s_i s_j}{s^2} +
\frac{s_{11ij}}{s_{11}} - \frac{s_{11i}s_{11j}}{s_{11}^2} +
s_{1j}s_{1i} + s_1s_{1ij} \right) \\
&=& \frac ns -\frac{2s^{ij}s_is_j}{s^2} +
\frac{s^{ij}s_{11ij}}{s_{11}} - \frac{s^{ij}s_is_1s_{1j}}s -
\frac{s^{ij}s_js_1s_{1i}}s \\
&&{}- s^{ij}s_1^2s_{1i}s_{1j} + s^{ij}s_{1j}s_{1i} +
s^{ij}s_1s_{1ij} \qquad \qquad \mbox{(by (\ref{grad-w-zero}))}\\
&=& \frac ns -\frac{2s^{ij}s_is_j}{s^2} +
\frac{s^{ij}s_{11ij}}{s_{11}} - \frac{2s_1^2}s  - s_1^2s_{11} +
s_{11} + s^{ij}s_1s_{1ij} \\
&& \qquad \qquad \mbox{(since $s_{ij}$ is diagonal at $P$)}\\
 &\ge& \frac ns -\frac{2s^{ij}s_is_j}{s^2} - \frac{2s_1^2}s
- s_1^2s_{11} + s_{11} + s^{ij}s_1s_{1ij} + \frac{n+2}{ s}  \\
&&{}- s_1s^{ij}s_{ij1}  + \frac{(s^{ij}s_{ij1})^2}{(n+2) s_{11}}
 + \frac{s^{ik}s^{jl}s_{kl1}s_{ij1}}{s_{11}} \\
&& \qquad \qquad \mbox{(by (\ref{u-11t}))}\\
 &\ge & \frac
{2n+2}s - 2 \sum_{i=1}^n\frac{s_i^2}{s^2s_{ii}} - \frac{2s_1^2}s -
s_1^2s_{11} + s_{11} + \sum_{i,j=1}^n
\frac{s_{ij1}^2}{s_{11}s_{ii}s_{jj}}
\end{eqnarray*}
by collecting terms, completing the square, and since $s_{ij}$ is
diagonal at $P$. Continue computing
\begin{eqnarray*}
0 &\ge& \frac {2n+2}s - 2\sum_{i=1}^n\frac{s_i^2}{s^2s_{ii}} -
\frac{2s_1^2}s - s_1^2s_{11} + s_{11}  + \frac{s_{111}^2}{s_{11}^3}
+
2\sum_{i=2}^n \frac{s_{11i}^2}{s_{11}^2s_{ii}} \\
&=& \frac {2n+2}s - \frac{2s_1^2}{s^2s_{11}} - \frac{2s_1^2}s -
s_1^2s_{11} + s_{11}  + \frac{s_1^2}{s_{11}s^2} + \frac{2s_1^2}s +
s_1^2s_{11}
\end{eqnarray*}
by (\ref{grad-w-zero}) and since $s_{ij}$ is diagonal at $P$.
Finally, collect terms so that
 $$0\ge
s_{11} + \frac{2n+2}s  + \frac1{s_{11}}
\left(-\frac{s_1^2}{s^2}\right)
$$
and multiply each side of the inequality by $s^2s_{11}e^{s_1^2}$ to
find a quadratic inequality $$w^2 + a w + b \le 0$$ for
$w=|s|s_{11}e^{\frac12 s_1^2}$ at $P$ the point in $\Omega$ at which
the maximum of $w$ is achieved. The coefficients $a$ and $b$ involve
only $n$, $s(P)$ and $s_1(P)$, and so there is an upper bound of $w$
on $\Omega$ depending on only these quantities.
\end{proof}

This bounds $s_{ij}$ away from infinity, which, together with
Andrews's speed estimate, shows that the ellipticity is locally
uniformly controlled in the interior of appropriate bowl-shaped
domains. In the next section, we use barriers essentially due to
Calabi \cite{calabi72} to ensure that appropriate bowl-shaped
domains exist, and so Guti\'errez-Huang's estimate applies.

\section{Barriers}
We will use two soliton solutions to the affine normal flow as inner
and outer barriers.  First of all, the unit sphere is a shrinking
soliton, and we use its affine images, ellipsoids, as inner
barriers.  Since the ellipsoids are compact, their support functions
are finite and smooth on all $\re^{n+1}$, and the usual maximum
principle applies:  If for an ellipsoid $E$, $s_E\le s_i$ on all
$\re^{n+1}$ (which is equivalent to the inclusion of convex hulls
$\hat E \subset \widehat{\mathcal L_i}$ for $\mathcal L_i$ the
hypersurface whose support function is $s_i$), then the maximum
principle for parabolic equations on $\mathbb S^n$ shows that
$s_E(t)\le s_i(t)$ for all positive $t$ before the extinction time
of $s_E(t)$.

The outer barrier we use is an expanding soliton due to Calabi
\cite{calabi72}.  Upon taking an affine transformation, its support
function $s_{\mathcal C}$ has $\mathcal D^\circ(s_{\mathcal C})$ an
open cone over a simplex, and has the value of a linear function
there. (Outside its domain, recall the support function is
$+\infty$.) Moreover, under the affine normal flow, $s_{\mathcal
C}(t)$ satisfies Dirichlet conditions on the boundary, and is
continuous and finite on the closure of its domain.  These
properties make Calabi's example very useful as an outer barrier (as
exploited by Cheng-Yau \cite{cheng-yau77,cheng-yau82} for the
elliptic real Monge-Amp\`ere equation).

Recall that $s_i\nearrow s$, where $s_i$ are the support functions
of strictly convex smooth compact hypersurfaces $\mathcal L_i$ which
approach $\mathcal L$. On $\mathcal D^\circ(s)$, as $s_i\nearrow s$
uniformly on compact subsets, and since the $s_i$ are convex, we
automatically have uniform $C^0$ and $C^1$ estimates on compact
subsets of $\mathcal D^\circ (s)$.  We define $s(t) = \lim_{i\to
\infty} s_i(t)$ for positive $t$ also, and so we have locally
uniform $C^0$ and $C^1$ estimates for positive $t$ as well.

To get similar uniform local ellipticity bounds for small positive
$t$, we need to check the hypotheses of Propositions
\ref{andrews-speed} and \ref{gh-est} as well.  For Proposition
\ref{andrews-speed}, we must ensure that $s_i(Y)\ge r$ for all large
$i$, $t\in [0,T]$, and $Y\in\sph^n$.  The affine normal flow of a
sphere provides a lower barrier to show this.  In particular, we
have the solution corresponding to the affine normal flow of a
sphere centered at the origin.  For any $r_0>0$, let
\begin{equation} \label{sphere-solution} u(Y,t) = r(t)|Y|, \qquad r(t) = \left( r_0^{\frac{2n+2}
{n+2}} - \frac {2n+2} {n+2} \, t \right) ^{\frac{n+2}{2n+2}}.
\end{equation}
Then $u$ satisfies the affine normal flow equation for a
support function.  Now the nondegeneracy assumption
(\ref{convex-s-Y}) shows that we can use the transformation law
(\ref{affine-transform}) with $A$ the identity matrix and
$b=-P$ to show $s(Y)\ge \epsilon$ for all $Y\in\sph^n$. Thus
(\ref{sphere-solution}) and the maximum principle show that for
$r=\epsilon/2$ there is a $T>0$ so that for all $t\in[0,T]$ and
$Y\in\sph^n$, and large $i$, we have $s_i(Y,t)\ge r$.  Thus we
can apply Andrews's estimate for all time in $t\in[0,T]$.

\begin{pro} \label{long-time}
Let $\mathcal L$ be a noncompact convex properly embedded
hypersurface in $\re^{n+1}$ which contains no lines.  Then the
affine normal flow $\mathcal L(t)$ exists for all positive time
$t>0$.
\end{pro}
\begin{proof}
 We will phrase this in terms of the support function.  Since
 $\mathcal L$ is noncompact, there is a ray $R=\{v+tw: t\ge0\}$ contained in the convex
 hull $\hat{\mathcal L}$.  We may choose coordinates so that $w = (0,1)\in
 \re^n\times \re$.  Therefore, the support function $$s(Y)=
 s_{\mathcal L}(Y) \ge s_R(Y) =
 \left\{ \begin{array}{c@{\quad \mbox{for}\quad}c} + \infty & y^{n+1}>0 \\
\langle v, Y \rangle & y^{n+1}\le 0
 \end{array} \right.$$
We will use this estimate, together with the nondegeneracy
assumption (\ref{convex-s-Y}) to provide a lower barrier.  In
particular, there is an $\epsilon>0$ so that $s(Y)=+\infty$ for
$y^{n+1}>0$ and $$s(Y) \ge \epsilon|Y| + \langle v,Y\rangle \quad
\mbox{for} \quad y^{n+1}\le0.$$

The barrier we will use is, for $y=(y^1,\dots,y^n)$ and
$Y=(y,y^{n+1})$, $$s_{E_j}(Y) = \epsilon \sqrt {|y|^2 +
(jy^{n+1})^2 } + \langle v,Y\rangle + jy^{n+1}.$$ This is the
support function of an ellipsoid centered at $P + (0,j)$ with
$n$ minor axes of length $\epsilon$ and one major axis of
length $\epsilon j$. Clearly for all $j>1$, $s_{E_j}(Y)\le
s(Y)$. As $j\to\infty$, the ellipsoid is equivalent, under a
volume-preserving affine map, to a sphere of radius $\epsilon
j^{\frac1{n+1}}$, which also goes to infinity. Now
(\ref{sphere-solution}) shows that the extinction time of the
ellipsoid under the affine normal flow goes to infinity as
$j\to\infty$.  Since the $s_{E_j}$ are all lower barriers to
$s$ (which is equivalent to the ellipsoids $E_j$ being inside
the convex hull $\hat{\mathcal L}$), we have that the affine
normal flow applied to $s$ must exist for all time.
\end{proof}

Now to find appropriate bowl-shaped domains to apply Proposition
\ref{gh-est}, we use an upper barrier due to Calabi.  This barrier
is first used in the real elliptic Monge-Amp\`ere equation by
Cheng-Yau \cite{cheng-yau77,cheng-yau82}. Calabi's example is based
on the fact that the hypersurface
$$\mathcal C(t) = \left\{(x^1,\dots,x^{n+1})\in\re^{n+1} : x_i>0, \,
\prod_{j=1}^{n+1} x^j = k>0 \right\}$$ is an expanding soliton for
the affine normal flow (which evolves by setting the parameter
$k=k(t)$ for an appropriate function). At time $t=0$, we set the
hypersurface
$$\mathcal C(0)= \left\{(x^1,\dots,x^{n+1})\in\re^{n+1} : x_i\ge0,
\, \prod_{j=1}^{n+1} x^j = 0 \right\}$$ the boundary of the first
orthant in $\R^{n+1}$. The support function of this example is given
for $c_n = (n+1)^{\frac12} ( \frac2{n+2} ) ^ {\frac{n+2}2}$ :
\begin{equation} \label{calabi-ex}
 s_{\mathcal C} (Y,t) = \left\{ \begin{array}{cl} + \infty & \mbox{
 if any } y^i > 0 \\
- (n+1) \left( c_n t^{\frac{n+2}n} \prod_{i=1}^{n+1} |y^i| \right)
^{\frac 1{n+1}} & \mbox{ if all } y^i\le 0
\end{array}
 \right.
\end{equation}
Note in particular that for time $t=0$, $s_{\mathcal C}(Y,0)$
is 0 on the closed orthant on which all the $y^i\le0$ and is
$+\infty$ elsewhere.  In order to find a more flexible class of
barriers, we can apply (\ref{affine-transform}) to transform
$s_{\mathcal C}$ by a volume-preserving affine map
$\Phi\!:x\mapsto Ax+b$ to be any linear function $\langle
b,Y\rangle$ on any linear image $(A^\top)^{-1} \mathcal C$, and
$+\infty$ elsewhere.  In our standard affine coordinates
$Y=(y,-1)$, we find that the support function of $\mathcal
C(0)$ can be transformed to have its domain be a simplex (this
is a projective image of the first orthant in $\re^n$), and the
value of $s_{\Phi \mathcal C}(0)$ is any affine function of $y$
on this domain.  The graphs of these functions will give us the
flexibility to create upper barriers for the support function
which ensure that the function $s$ does move by a certain
amount under the affine normal flow. This in turn gives a
bowl-shaped domain in which to apply Guti\'errez-Huang's
interior estimates for the Hessian of $s$.

Assume that the domain $\mathcal D^\circ(s)$ is contained in
the lower half-space of $\re^{n+1}$. So since $s$ has
homogeneity one, $s$ can be described by its behavior on the
affine hyperplane $\mathcal H = \{(y,-1) : y\in\re^n\}$.  For
the remainder of this section, we consider the domain $\mathcal
D^\circ(s)$ to be a subset of $\re^n$, as identified with the
affine plane $\mathcal H$.

Each $x\in \mathcal D^\circ(s)$ has a convex neighborhood
$\mathcal N$ on which $s_i \to s$ uniformly as an increasing
sequence of convex functions, and so that the Lipschitz norms
$\|s_i\|_{C^{0,1}} (\mathcal N)$ are bounded by a constant $C$
independent of $i$.  By adding linear functions (constant in
$t$) to the $s_i$, we may assume $s_i(x)=0$ and $\nabla
s_i(x)=0$. This normalization does not affect the
Monge-Amp\`ere equation (\ref{ma-eq}) or the Hessian of $s_i$
(and so the $C^2$ estimates we derive apply to the original
$s_i$ as well). We can choose points $p_1,\dots p_{n+1}$ so
that $0\le s_i(y) \le C'$ for $C'$ a constant independent of
$i$ and $y$ in the convex hull $\mathcal Q$ of the $p_j$.  We
may also assume that $x$ is in the interior of $\mathcal Q$.
Now consider the simplices $\mathcal S_j$ to be the convex hull
in $\re^n$ of the points
$$x,p_1,\dots,p_{j-1},\widehat{p_j}, p_{j+1}, \dots,
p_{n+1},$$ where $p_j$ is omitted from the list.  Define $P_j$
to be an affine function on each $\mathcal S_j$ which is equal
to $C'$ on each of the $p_k\in \mathcal S_j$ and is equal to 0
at $x$, and define $P_j$ to be $+\infty$ outside $\mathcal
S_j$. Then define $P(y) = \min_j P_j(y)$. Then it is clear that
$P$ is satisfies $P_j(y) \ge P(y)\ge s_i(y)$ for all $i$ and
for all $y\in\re^n$.

We do not know the explicit solution to the Monge-Amp\`ere
equation (\ref{ma-eq}) with initial value $P$, but all we need
to show to produce uniformly large bowl-shaped domains centered
at $x$ for each of the $s_i$ is that $P(x,t)<0$ for positive
$t$.  This can be verified as follows: By the discussion above,
$P_j$ is the image of Calabi's example $\mathcal C(0)$ under an
affine transformation $z\mapsto A z + b$ of $\re^{n+1}$. By the
explicit solution (\ref{calabi-ex}) and the transformation law
(\ref{affine-transform}), we see that $P(t,y)\le P_j(t,y)<0$
for small $t>0$ and all $y$ near $x$ on the ray from $x$ to the
barycenter of $\mathcal S_j$.  Therefore, since $P(t,y)$ is
convex in $y$ and $x$ is in the convex hull of the barycenters
of the $\mathcal S_j$, we have shown that $P(t,x)<0$ for all
small positive $t$.

By the maximum principle, each sub-level set of each $s_i$ contains
a sub-level set of $P$, which shows that  $x\in \mathcal D^\circ(s)$
has a uniformly large bowl-shaped domain around it for each $s_i$
independently of $i$.  So Guti\'errez-Huang's Hessian estimates are
uniform in every compact subset of $\mathcal D^\circ (s) \times
(0,T]$ for small $T$.

By standard techniques, both Guti\'errez-Huang's and Andrews's
estimates can be extended in time to be uniform in compact subsets
of $\mathcal D^\circ(s)\times(0,\infty)$.  These estimates uniformly
control the spacelike $C^2$ norm and the ellipticity of $s_i$.  Then
the Monge-Amp\`ere equation allows us to apply Krylov's regularity
theory to get local uniform $C^{2+\alpha,1+\alpha/2}$ estimates,
which can then be bootstrapped to show

\begin{thm} \label{smooth-conv}
On $\mathcal D^\circ(s)\times (0,\infty)$, $s_i\to s$ in the
$C^\infty_{\rm loc}$ topology.
\end{thm}

Also note that in \cite{loftin-tsui06} we use the same inner and
outer barriers to show
\begin{pro}
 Under the affine normal flow, $s$ satisfies a Dirichlet boundary
 condition on $\partial \mathcal D(s)$.
\end{pro}
This proposition holds regardless of the boundary regularity---$s$
can be infinite or finite and discontinuous on the boundary
$\partial \mathcal D(s)$ \cite{rockafellar}.  We also use the
barriers to show
 \begin{pro}
For every $t>0$, $F=F(y,t)$ is properly embedded as a function of
$y$ for $(y,-1)\in \mathcal D^\circ(s)$.  In other words, as
$(y,-1)\to \partial\mathcal D^\circ(s)$, at least one coordinate of
$$F(y) = (s_1(y),\dots , s_n(y), s_k(y) y^k - s(y))$$ goes to
$\pm\infty$.
 \end{pro}

\section{The evolution of $|C|^2$}
Here we recall an estimate of Andrews \cite{andrews96} on the
evolution of $|C|^2 = g^{il}g^{jm}g^{kp} C_{ijk} C_{lmp}$.  For
a compact strictly convex initial hypersurface evolving under
the affine normal flow, $$\left(\partial_t - \frac1{n+2}
\Delta\right) |C|^2 \le - \frac2{n(n+2)} |C|^4.$$ Then the
maximum principle shows that for all $t\in(0,T)$ for $T$ the
extinction time, \begin{equation} \label{C2-decay} |C|^2 \le
\frac{n(n+2)} {2t}
\end{equation}
 independently of initial conditions.

Since Theorem \ref{smooth-conv} above shows that $s_i\to s$ in
$C^\infty_{\rm loc}$ on $\mathcal D^\circ(s) \times (0,\infty)$, the
pointwise bound (\ref{C2-decay}) survives in the limit for any
solution to the affine normal flow beginning at time $t=0$. If the
flow begins at time $\tau$ instead, then of course we have $$|C|^2
\le \frac{n(n+2)} {2(t-\tau)},$$ and for an ancient solution
($\tau\to-\infty$), we must have $|C|^2=0$. In the following
section, we give a proof of the classical theorem of Berwald that
says that $C_{ijk}=0$ implies the hypersurface is quadric.  Thus any
ancient solution to the affine normal flow must be a quadric
hypersurface.  Since a hyperboloid cannot form part of an ancient
solution, we have

\begin{thm} \label{ancient-sol-thm}
Any ancient solution to the affine normal flow is a paraboloid
or an ellipsoid.
\end{thm}

\section{Quadric Hypersurfaces} \label{quadric-sec}
Now we prove a classical theorem of Berwald, that the cubic form
$C_{ijk}=0$ implies that the hypersurface is a quadric.  The first
step is to show that the hypersurface is an affine sphere (i.e.,
that $\xi = a F + V$ for a constant scalar $a$ and a constant vector
$V$).

Compute for $C_{ijk}=0$
\begin{equation} \label{c=0}s_{ijk}=\frac1{n+2} \Big(
s_{ij}(\ln D)_k +
 s_{jk}(\ln D)_i+s_{ki}(\ln D)_j\Big)\end{equation}
and differentiate to find
 \begin{equation}
\begin{split} & (n+2)s_{ijkl}\\
= \ & \Big( s_{ijl}(\ln D)_k + s_{ij}(\ln D)_{kl}+
 s_{jkl}(\ln D)_i +s_{jk}(\ln D)_{il} +s_{kil}(\ln D)_j +
s_{ki}(\ln D)_{jl}\Big)\\
= \ &  \frac{1}{n+2} (s_{ij}(\ln D)_l(\ln D)_k+ s_{jl}(\ln D)_i(\ln
D)_k+s_{li}(\ln D)_j(\ln D)_k)+s_{ij}(\ln D)_{kl}\\
+ \ & \frac{1}{n+2} (s_{jk}(\ln D)_l(\ln D)_i+ s_{kl}(\ln D)_j(\ln
D)_i+s_{lj}(\ln D)_k(\ln D)_i)+s_{jk}(\ln D)_{il}\\
+ \ & \frac{1}{n+2} (s_{ki}(\ln D)_l(\ln D)_j+ s_{il}(\ln D)_k(\ln
D)_j+s_{lk}(\ln D)_i(\ln D)_j)+s_{ki}(\ln D)_{jl}
\end{split}
 \end{equation}

\begin{equation}
\begin{split} & (n+2)s_{ilkj}\\
= \ &  \frac{1}{n+2} (s_{il}(\ln D)_j(\ln D)_k+ s_{lj}(\ln D)_i(\ln
D)_k+s_{ji}(\ln D)_l(\ln D)_k)+s_{il}(\ln D)_{kj}\\
+ \ & \frac{1}{n+2} (s_{lk}(\ln D)_j(\ln D)_i+ s_{kj}(\ln D)_l(\ln
D)_i+s_{jl}(\ln D)_k(\ln D)_i)+s_{lk}(\ln D)_{ij}\\
+ \ & \frac{1}{n+2} (s_{ki}(\ln D)_j(\ln D)_l+ s_{ij}(\ln D)_k(\ln
D)_l+s_{jk}(\ln D)_i(\ln D)_l)+s_{ki}(\ln D)_{lj}
\end{split}
 \end{equation}

Using $s_{ijkl}=s_{ilkj}$, we have
\begin{equation}
\begin{split} & s_{ij}( (\ln D)_{kl}-\frac{1}{n+2}(\ln D)_k(\ln
D)_l)+s_{jk}( (\ln D)_{li}-\frac{1}{n+2}(\ln D)_l(\ln
D)_i) \\
= \ & s_{il}( (\ln D)_{kj}-\frac{1}{n+2}(\ln D)_k(\ln D)_j)+s_{lk}(
(\ln D)_{ij}-\frac{1}{n+2}(\ln D)_i(\ln D)_l)
\end{split}
 \end{equation}

Multiplying $s^{ij}$ to previous equation, we get
\begin{equation}
\begin{split} & n ( (\ln D)_{kl}-\frac{1}{n+2}(\ln D)_k(\ln D)_l)+( (\ln
D)_{lk}-\frac{1}{n+2}(\ln D)_l(\ln D)_k)\\
 & = ( (\ln D)_{kl}-\frac{1}{n+2}(\ln D)_k(\ln D)_l)+s_{lk}s^{ij}(
(\ln D)_{ij}-\frac{1}{n+2}(\ln D)_i(\ln D)_l)
\end{split}
 \end{equation}
So $$n ( (\ln D)_{kl}-\frac{1}{n+2}(\ln D)_k(\ln D)_l)=s_{lk}s^{ij}(
(\ln D)_{ij}-\frac{1}{n+2}(\ln D)_i(\ln D)_l)$$

 Let $S$ be the  matrix $(s_{ij})$ and $T$ be the  matrix with
 $T_{ij}=(\ln D)_{ij}- \frac{ (\ln D)_i(\ln D)_j}{n+2}$.
 So we have $T = \frac{g^{ij}T_{ij}}{n}S$. Denote $\mbox{tr}\, T
 =g^{ij}T_{ij}$.

From $(n+2)\xi = D^{-\frac{1} {n+2}} ( (\ln D)_1, \dots,(\ln D)_n,
(n+2) + (\ln D)_i y^i )$. So for $\xi^i$ the $i^{\rm th}$ component
of $\xi$, \begin{equation}
\begin{split} &(n+2)\partial_j(\xi^i)\\
= \ & \partial_j (D^{-\frac{1} {n+2}}(\ln D)_i)\\
= \ & -\frac{1} {n+2}D^{-\frac{1} {n+2}}(\ln D)_j(\ln
D)_i+D^{-\frac{1} {n+2}}(\ln
D)_{ij}\\
= \ & D^{-\frac{1} {n+2}}T_{ij}= \frac{D^{-\frac{1} {n+2}} \,
\mbox{tr}\, T}{n}s_{ij}\end{split}
 \end{equation}for $ 1\le i \le n$.
Similarly, $$(n+2)\partial_j(\xi^{n+1})=D^{-\frac{1} {n+2}}
 ((\ln D)_{ij} - \frac{1} {n+2}(\ln D)_i(\ln
D)_j)y^i=D^{-\frac{1} {n+2}}T_{ij}y^i$$ $$= \frac{D^{-\frac{1}
{n+2}}\,\mbox{tr}\, T}{n}s_{ij}y^i$$ Therefore
$\xi_{,i}=\frac{D^{-\frac{1} {n+2}}\,\mbox{tr}\, T}{n}F_i$

Recall that $F_i = (s_{1i},\dots, s_{ni}, s_{li} y^l)$. We have
$\xi_{,i}=\frac{D^{-\frac{1} {n+2}}\,\mbox{tr}\, T}{n}F_i$. Affine
curvature is defined by $ \xi_{,i} = -A^k_i F_{,k}$. So $-A^k_i=
\frac{D^{-\frac{1} {n+2}}\,\mbox{tr}\, T}{n}\delta^k_i= a
\delta^k_i$ where $a= \frac{D^{-\frac{1} {n+2}}\,\mbox{tr}\, T}{n}$

Now the affine structure equations, applied to the second ordinary
derivative $\xi_{ij}$, shows
 \begin{eqnarray*}
 \xi_{ij} &=& (aF_i)_j \\
 &=& a_j F_i + aF_{ij} \\
 &=& a_j F_i + a(g_{ij} \xi + (\Gamma_{ij}^k  +C_{ij}^k)F_k \\
 &=& (a_j \delta^k_i + a\Gamma_{ij}^k)F_k + ag_{ij}\xi.
 \end{eqnarray*}
So $a_j\delta^k_i$ must be symmetric in $i,j$, and in particular,
$a_i \delta^k_k=a_k\delta^k_i$. Since $ n \ge 2$, we have $a_i=0$
for all $ i$. So $a$ is constant and $\xi_k=aF_k$ implies that $\xi
= a F + V$, where $V$ is a constant vector.

So far, we have shown
\begin{pro} \label{C-vanish-aff-sphere}
Let $n\ge2$. If $C_{ijk}=0$ then $\xi = a F + V$ for $V$ a constant
vector and $a$ a constant scalar.
\end{pro}

The rest of the proof of the following theorem follows Nomizu-Sasaki
\cite{nomizu-sasaki}.

\begin{thm} \label{cubic-quadric}
Assume $n\ge2$. If the cubic form $C_{ijk}=0$, then the hypersurface
given by the image of $F$ is a quadric hypersurface.  In other
words, there is a second-degree polynomial map $\mathcal P\!:
\re^{n+1} \to \re$ so that $\mathcal L$ is an open subset of
$\{\mathcal P=0\}$.
\end{thm}

\begin{proof}
Let $\mathcal L$ denote our hypersurface with is (locally) the image
of the embedding $F$.  For each $x= F(y) \in \mathcal L$, since
$\{F_1,\dots,F_n,\xi\}$ is a basis of $\R^{n+1}$, we can write each
point $P\in\re^n$ uniquely as
\begin{equation} \label{mu-U-def}
  P = F(y) + U^i_P(y)
F_i(y) + \mu_P(y) \xi(y).
\end{equation}
Then the \emph{Lie quadric} of $\mathcal L$ at $x=F(y)$ is defined
as the locus
$$ \mathcal F_y = \{P\in\re^{n+1} : g_{ij} U^i U^j - a  \mu^2 -
2\mu = 0\},$$ where $a$ is the constant determined in Proposition
\ref{C-vanish-aff-sphere} above and $g_{ij}=g_{ij}(y)$.  For each
$y$, $\mathcal F_y$ is clearly a quadric hypersurface in
$\re^{n+1}$.

Now we will show that for each $x\in \mathcal L$, that $\mathcal L
\subset \mathcal F_x$.  By dimension considerations, this show that
$\mathcal L$ is an open subset of the quadric $\mathcal F_x$, and we
are done.  Now consider $y_0$ for $F(y_0)=P\in\mathcal L$, and
consider $U^i$ and $\mu$ defined in (\ref{mu-U-def}) above as
functions of $y$ with $y_0$ fixed.  Now differentiate
(\ref{mu-U-def}) to find for $k=1,\dots,n$ and $U^i_k = \partial_k
U^i$,
$$ 0 = \partial_k P = U^i_k F_i + U^i F_{ik} + \mu_k \xi + \mu
\xi_k.$$
 By Proposition \ref{C-vanish-aff-sphere}, $\xi_k = a F_k$,
and also $F_{ik} = (\Gamma^j_{ik} + C^j_{ik}) F_j + g_{ik}\xi$ for
$C^j_{ik}$ the cubic form and $\Gamma^j_{ik}$ the Levi-Civita
connection with respect to the affine metric $g_{ij}$.  Since we
assume the cubic form is zero, we have $F_{ik} = \Gamma^j_{ik} F_j +
g_{ik}\xi$.  Thus
$$ 0 = U^i_k F_i + U^i(\Gamma^j_{ik} F_j + g_{ik}\xi) + \mu_k \xi +
\mu a F_k,$$ and by splitting into the components on the basis
$\{F_1,\dots, F_n, \xi\}$, we find
 \begin{eqnarray} \label{Uik-eq}
 U^j_k &=& - U^i \Gamma^j_{ik} - (1+a\mu)\delta^j_k \quad \mbox{for }
 j,k=1,\dots,n,  \\
\label{mu-k-eq}
 \mu_k &=& - U^ig_{ik} \quad \mbox{for }k=1,\dots,n.
 \end{eqnarray}

 Now define $\Phi\!:\mathcal L\to \re$ by $$\Phi(y) = g_{ij} U^i U^j
 - a \mu^2 - 2\mu = D^{\frac1{n+2}} s_{ij} U^i U^j - a \mu^2 -
 2\mu.$$
 Note $\Phi(y_0)=0$ since by definition $U^i(y_0)=\mu(y_0)=0$.  So
if we show $\Phi_k = 0$, then $\Phi(y)=0$ for all $y$.  By the
definitions of $\Phi,U^i,\mu$, then we will have shown $y_0\in
\mathcal F_y$ and so $\mathcal L \subset \mathcal F_y$.

So in order to complete the proof of the theorem, we must check
$\Phi_k=0$. So compute, using (\ref{Uik-eq}) and (\ref{mu-k-eq})
above,
\begin{eqnarray*}
 \Phi_k &=& \frac1{n+2} \, D^{\frac1{n+2}} (\ln D)_k s_{ij} U^i U^j
 +D^{\frac1{n+2}} s_{ijk} U^iU^j + 2D^{\frac1{n+2}} s_{ij} U^i_k U^j
\\&&{} - 2a\mu \mu_k -2\mu_k \\
&=& \frac1{n+2} \, D^{\frac1{n+2}} (\ln D)_k s_{ij} U^i U^j +
D^{\frac1{n+2}} s_{ijk} U^iU^j  + 2(a\mu+1)U^i g_{ik} \\
&&{}+2D^{\frac1{n+2}} s_{ij} U^j[ -U^l \Gamma^i_{lk} - (1+a\mu)
\delta^i_k] \\
&=& \frac1{n+2} \, D^{\frac1{n+2}} (\ln D)_k s_{ij} U^i U^j +
D^{\frac1{n+2}} s_{ijk} U^iU^j - 2D^{\frac1{n+2}} s_{ij} U^j U^l
\Gamma^i_{lk} \\
&=& \frac1{n+2} \, D^{\frac1{n+2}} (\ln D)_k s_{ij} U^i U^j +
D^{\frac1{n+2}} s_{ijk} U^iU^j - D^{\frac1{n+2}} s_{ij} \left[
\frac1{n+2} (\ln D)_k \delta^i_l \right. \\
&& \left. {}+ \frac1{n+2} (\ln D)_l \delta^i_k + s^{im} s_{lkm} -
\frac1{n+2} s^{im} (\ln D)_m s_{lk} \right] \\
&=&0.
\end{eqnarray*}
This completes the proof of Theorem \ref{cubic-quadric}.
\end{proof}

\bibliographystyle{abbrv}
\bibliography{thesis}

\begin{thebibliography}{10}

\bibitem{andrews96}
B.~Andrews.
\newblock Contraction of convex hypersurfaces by their affine normal.
\newblock {\em Journal of Differential Geometry}, 43(2):207--230, 1996.

\bibitem{andrews00}
B.~Andrews.
\newblock Motion of hypersurfaces by {G}auss curvature.
\newblock {\em Pacific J. Math.}, 195(1):1--34, 2000.

\bibitem{calabi72}
E.~Calabi.
\newblock Complete affine hyperspheres {I}.
\newblock {\em Instituto Nazionale di Alta Matematica Symposia Mathematica},
  10:19--38, 1972.

\bibitem{cheng-yau76}
S.~Y. Cheng and S.~T. Yau.
\newblock On the regularity of the solution of the {$n$}-dimensional
  {M}inkowski problem.
\newblock {\em Comm. Pure Appl. Math.}, 29(5):495--516, 1976.

\bibitem{cheng-yau77}
S.-Y. Cheng and S.-T. Yau.
\newblock On the regularity of the {M}onge-{A}mp{\`e}re equation
  $\det((\partial^2u/\partial x^i\partial x^j))={F}(x,u)$.
\newblock {\em Communications on Pure and Applied Mathematics}, 30:41--68,
  1977.

\bibitem{cheng-yau82}
S.-Y. Cheng and S.-T. Yau.
\newblock On the real {M}onge-{A}mp{\`e}re equation and affine flat structures.
\newblock In {\em Proceedings of the 1980 {B}eijing Symposium on Differential
  Geometry and Differential Equations, Vol. 1,2,3 ({B}eijing, 1980)}, pages
  339--370. Science Press, 1982.

\bibitem{chow85}
B.~Chow.
\newblock Deforming convex hypersurfaces by the {$n$}th root of the {G}aussian
  curvature.
\newblock {\em J. Differential Geom.}, 22(1):117--138, 1985.

\bibitem{gutierrez-huang98}
C.~E. Guti{\'e}rrez and Q.~Huang.
\newblock A generalization of a theorem by {C}alabi to the parabolic
  {M}onge-{A}mp\`ere equation.
\newblock {\em Indiana Univ. Math. J.}, 47(4):1459--1480, 1998.

\bibitem{hamilton86}
R.~S. Hamilton.
\newblock Four-manifolds with positive curvature operator.
\newblock {\em J. Differential Geom.}, 24(2):153--179, 1986.

\bibitem{loftin-tsui06}
J.~Loftin and M.-P. Tsui.
\newblock Ancient solutions of the affine normal flow.
\newblock to appear, Journal of Differential Geometry, math.DG/0602484.

\bibitem{nomizu-sasaki}
K.~Nomizu and T.~Sasaki.
\newblock {\em Affine Differential Geometry: Geometry of Affine Immersions}.
\newblock Cambridge University Press, 1994.

\bibitem{rockafellar}
R.~T. Rockafellar.
\newblock {\em Convex analysis}.
\newblock Princeton Landmarks in Mathematics. Princeton University Press,
  Princeton, NJ, 1997.
\newblock Reprint of the 1970 original, Princeton Paperbacks.

\bibitem{tso85}
K.~Tso.
\newblock Deforming a hypersurface by its {G}auss-{K}ronecker curvature.
\newblock {\em Comm. Pure Appl. Math.}, 38(6):867--882, 1985.

\end{thebibliography}

 \end{document}